\documentclass[a4paper,12pt,twoside]{article}
\usepackage[latin1]{inputenc}
\usepackage[T1]{fontenc}
\usepackage[english]{babel}
\usepackage{amsmath,amssymb,amscd}
\usepackage{amsthm}
\usepackage{array}
\usepackage{hyperref}
\usepackage[pdftex]{color}
\usepackage{enumerate}
\usepackage{enumitem}
\usepackage{comment}
\usepackage{cancel}
\newtheorem{definition}{Definition}[section]
\newtheorem{remark}[definition]{Remark}

\newtheorem{corollary}[definition]{Corollary}
\newtheorem{lemma}[definition]{Lemma}
\newtheorem{proposition}[definition]{Proposition}
\newtheorem{example}[definition]{Example}

\newtheorem{question}[definition]{Question}

\def\Ind#1#2{#1\setbox0=\hbox{$#1x$}\kern\wd0\hbox to 0pt{\hss$#1\mid$\hss}
\lower.9\ht0\hbox to 0pt{\hss$#1\smile$\hss}\kern\wd0}
\def\ind{\mathop{\mathpalette\Ind{}}}
\def\Notind#1#2{#1\setbox0=\hbox{$#1x$}\kern\wd0\hbox to 0pt{\mathchardef
\nn=12854\hss$#1\nn$\kern1.4\wd0\hss}\hbox to 
0pt{\hss$#1\mid$\hss}\lower.9\ht0
\hbox to 0pt{\hss$#1\smile$\hss}\kern\wd0}
\def\nind{\mathop{\mathpalette\Notind{}}}

\title{Foundation ranks and supersimplicity}
\author{Santiago C\'ardenas-Mart\'in, Rafel Farr\'e}
\date{Version 2022-02-15}

\begin{document} 

\maketitle
\thispagestyle{empty}

\bigskip
\begin{abstract}
	We introduce a new foundation rank based in the relation of dividing between partial types.
	We call $DU$ to this rank.
	We also introduce a new way to define the $D$ rank over formulas as a foundation rank.
	In this way, $SU$, $DU$ and $D$ are foundation ranks based in the relation of dividing.
	We study the properties and the relations between these ranks.

	Next, we discuss the possible definitions of a supersimple type.
	This is a notion that it is not clear in the previous literature. 
	In this paper we give solid arguments to set up a concrete definition of this notion and its properties.
	We also see that $DU$ characterizes supersimplicity, while $D$ not.
\end{abstract}

\section{Conventions}
    We denote by $L$ a language and $T$ a complete theory. 
    We denote by $\mathfrak{C}$ a monster model of $T$,  that is a $\kappa$-saturated and strongly
      $\kappa$-homogeneous model for a cardinal $\kappa$ large enough.
    Models $M,N,\ldots$ are considered elementary substructures of $\mathfrak{C}$ with cardinal less than $\kappa$
      and every set of parameters $A,B,\ldots$ is considered as a subset of $\mathfrak{C}$ with cardinal less than
      $\kappa$.

    We denote by $a,b,\ldots$ tuples of elements of the monster model, possibly infinite (of length less than
      $\kappa$).
    We often use these tuples as ordinary sets regardless of their order. 
    We often omit union symbols for sets of parameters, for example we write $ABc$ to mean $A\cup B\cup c$.
    Given a sequence of sets $( A_i : i\in\alpha )$ we use $A_{<i}$ and $A_{\leq i}$  to denote $\bigcup_{j<i}A_j$
      and $\bigcup_{j\leq i}A_j$ respectively.
    We use $I$ to denote a infinite index set without order and use $O$ for a infinite lineal ordered set.
    Unless otherwise stated, all the types are finitary.
    We use $\ind^d$ and $\ind^f$ to denote the independence relations for non-dividing and non-forking respectively.
    By $dom(p)$ we denote the set of all parameters that appear in some formula of $p$.

\section{The DU-rank}
	We are going to introduce a new rank that we call $DU$. 
	$DU$ is a foundation rank as it is the known rank $SU$ (for their definitions and properties, see for example,
	  Casanovas\cite{Casanovas11}).  
  	We will define the rank $DU$ as the foundation rank of the relation of dividing between pairs $(p,A)$ of partial
	  types and set of parameters satisfying $dom(p)\subseteq A$.
	Similarly we will define  $DU^f$ using the relation of forking, although we will check a little later
	  (Proposition~\ref{DUeqDUf}) that both ranks are the same.

	Let us begin by remembering the notion of foundation rank:
	\begin{definition}
		Let $R$ be a binary relation defined in a set or class of mathematical objects. 
		The foundation rank of $R$ is the mapping $r$ assigning to every element $a$ of the domain of $R$ an ordinal
		  number or $\infty$ according to the following rules:
		\begin{enumerate}
			\item $r(a)\geq 0$.
			\item $r(a)\geq\alpha+1$ if and only if there exists $b$ such that $aRb$ and $r(b)\geq\alpha$.
			\item $r(a)\geq\alpha$ with $\alpha$ a limit ordinal, if and only if $r(a)\geq\beta$ for all
			  $\beta<\alpha$.
		\end{enumerate}
		One defines $r(a)$ as the supremum of all $\alpha$ such that $r(a)\geq\alpha$.
		If such supremum does not exist we set $r(a)=\infty$.
	\end{definition}

    Now, we define $DU$ and $DU^f$ and we will check that really $DU$ does not depend of the set of  parameters. 
    We denote provisionally by  $DU(p,A)$ the $DU$ rank of the pair $(p,A)$.
	\begin{definition}
	    $DU$ and $DU^f$ are the foundation ranks of the following relations $R_d$ and $R_f$:
	    \begin{itemize}
			\item $(p(x),A)R_d(q(x),B)\text{ if and only if }p(x)\subseteq q(x)\text{ and }q\text{ divides over }A$
			\item $(p(x),A)R_f(q(x),B)\text{ if and only if }p(x)\subseteq q(x)\text{ and }q\text{ forks over }A$
		\end{itemize}	
	    where $p$ is a partial type over $A$ and $q$ is a partial type over $B$.
	\end{definition}

	\begin{remark}
		It is immediate to verify by induction that both ranks are invariant under conjugation (automorphism). 
	\end{remark}

	\begin{lemma}\label{IncDiv}
		Let $p(x)$ be a partial type dividing over $A$.
		Let $B\supseteq A$.
		Then, there exists $f\in Aut(\mathfrak{C}/A)$ such that $p^f$ divides over $B$.
	\end{lemma}
	\begin{proof} 
		Let $p(x)=q(x,a)$ for some  $q(x,y) $ without parameters and $a\subseteq A$.
		For $\lambda$ big enough there exist a set $\{a_i:i\in\lambda\}$ such that $a_i\equiv_A a$ for any
		  $i\in\lambda$ and $\bigcup_{i\in\lambda}q(x,a_i)$ is $k$-inconsistent.
		So, we can choose an infinite subset all having the same type over $B$, witnessing division over $B$.
	\end{proof}
	
	\begin{proposition}
		The rank $DU$ does not depend on the set of parameters $A$.
		That is, if $p(x)$ is a partial type with parameters in $A\cap B$  then $DU(p,A)=D(p,B)$.
		So, from now on we will use the notation $DU(p)$.
	\end{proposition}
	\begin{proof}
		It suffices to prove that given $p$ be a partial type over $A$ and  $A'\supseteq A$ then $DU(p,A)=DU(p,A')$.
		Obviously $DU(p,A)\geq DU(p,A')$. 
		For the proof of $DU(p,A)\leq DU(p,A')$, we show, by induction on $\alpha$, $DU(p,A)\geq\alpha$ implies
		  $DU(p,A')\geq\alpha$. 
		
	    If $DU(p,A)\geq\alpha+1$ then there exists $q\supseteq p$ over $B$ such that $q$ divides over $A$ and
	      $DU(q,B)\geq\alpha$.
		By the previous lemma, there exists an $A$-automorphism $f$ such that $q^f$ divides over $A'$.
		Then, $DU(q^f,B^f)\geq\alpha$. 
		By the induction hypothesis, $DU(q^f,A'B^f)\geq\alpha$.
		As $p\subseteq q^f$ and $q^f$ divides over $A'$, $DU(p,A')\geq\alpha+1$.
	\end{proof}

\section{Properties of the DU-rank}
    We begin by setting some basic properties of $DU$.
    From the first property, it follows that two equivalent partial types have identical $DU$-rank.
    So, the $DU$-rank of a type-definable set makes sense.
	
    \begin{remark}\label{DUBasic}
		Let $p(x),q(x)$ be partial types.
		\begin{enumerate}
			\item If $p\vdash q$ then $DU(p)\leq DU(q)$.
			\item $DU(p\vee q)=\max(DU(p),DU(q))$.
			\item $DU(p)=0$ if and only if $p$ is algebraic.
			\item Two type-definable sets with a definable bijection between them have the same $DU$-rank.
		\end{enumerate}
	\end{remark}
	\begin{proof}
		We Assume that $p$ and $q$ are over the same set of parameters $A$.

		\emph{1}.  
		    We  prove $DU(p)\geq\alpha$ implies $DU(q)\geq\alpha$ by induction on $\alpha$.
		    Assume $DU(p)\geq\alpha+1$.
		    Then, there exists $p_1\supseteq p$ such that $p_1$ divides over $A$ and $DU(p_1)\geq\alpha$.
		    Now $p_1\cup q$ extends $q$, divides over $A$,  and by the inductive hypothesis, 
		      $DU(p_1\cup q)\geq\alpha$.
			Therefore, $DU(q)\geq\alpha+1$.
			
		\emph{2}. 
		    By the previous point, $DU(p\vee q)\geq\max(DU(p),DU(q))$.
		    The other inequality is done by induction on $\alpha$.
		    Assume $DU(p\vee q)\geq\alpha+1$.
 	        There exists $r(x)\supseteq p(x)\vee q(x)$ such that $r$ divides over $A$ and $DU(r)\geq\alpha$.
	        By inductive hypothesis,  as  $r\equiv(p\cup r)\vee(q\cup r)$, $DU(p\cup r)\geq\alpha$ or 
	          $DU(q\wedge r)\geq\alpha$.
		    So, $DU(p)\geq\alpha+1$ or $DU(q)\geq\alpha+1$.
			
		\emph{3}. 
		    $DU(p)\geq1$ iff $p$ has some extension dividing over $A$ iff $p$ is non-algebraic.
		\emph{4}.  
		    Let $p(x),q(y)$ be partial types and let $f:p(\mathfrak{C})\to q(\mathfrak{C})$ be a definable
		      bijection. 
		    We assume $p,q$ are over $A$ and $f$ is defined over $A$. We prove by induction that
		      $DU(p(\mathfrak{C}))\geq\alpha$ implies $DU(q(\mathfrak{C}))\geq\alpha$. 
		    If  $DU(p(\mathfrak{C}))\geq\alpha+1$ there is some $p'(x)\supseteq p(x)$ such that $p'$ divides over
		      $A$ and $DU(p'(\mathfrak{C}))\geq\alpha$.
		    Then $f(p'(\mathfrak{C}))$ is type-definable and, by inductive hypothesis, 
		      $DU(f(p'(\mathfrak{C})))\ge \alpha$. 
		    It is not difficult to prove that if $q'(y)$ type-defines $f(p'(\mathfrak{C}))$ then $q(y)$ divides over
		      $A$.

	\end{proof}

    We are going to see some equivalences for $DU$:
	\begin{proposition}\label{EqvsDU}
		Let $p(x)$ be a partial type over a set of parameters $A$ and $\alpha$ an ordinal. 
		Denote $\mu=\left(2^{|T|+|A|}\right)^+$.
		The following are equivalent:
		\begin{enumerate}
			\item $DU(p)\geq\alpha+1$.

			\item There are $\psi(x,y)\in L$ and a countable sequence $(a_i:i<\omega)$ such that
				\begin{enumerate}
					\item $(a_i:i<\omega)$ is $A$-indiscernible.
					\item $\{\psi(x,a_i):i<\omega\}$ is inconsistent.
					\item For every $i<\omega$, we have $DU(p(x)\cup\{\psi(x,a_i)\})\geq\alpha$.
				\end{enumerate}

			\item There are $\psi(x,y)\in L$ and a number $k\geq2$ such that for every cardinal $\lambda$, 
			  there is a sequence $(a_i:i<\lambda)$ such that
				\begin{enumerate}
					\item $\{\psi(x,a_i):i<\lambda\}$ is $k$-inconsistent.
					\item For every $i<\lambda$, we have $DU(p(x)\cup\{\psi(x,a_i)\})\geq\alpha$.
				\end{enumerate}

			\item There are $\psi(x,y)\in L$, a number $k\geq2$ and a sequence $(a_i:i<\mu)$ such that
				\begin{enumerate}
					\item $\{\psi(x,a_i):i<\mu\}$ is $k$-inconsistent.
					\item For every $i<\mu$, we have $DU(p(x)\cup\{\psi(x,a_i)\})\geq\alpha$.
				\end{enumerate}

			\item There are a partial type $p'(x,y)$ over $\emptyset$ with $|y|\le|A|+|T|$, a number $k\geq2$ and a
			  sequence $(a_i:i<\mu)$ such that
				\begin{enumerate}
					\item The union of any $k$ types in $(p'(x,a_i):i\in\mu)$ is inconsistent.
					\item $p'(x,a_i)\vdash p(x)$ for each $i<\mu$.
					\item $DU(p'(x,a_i))\geq\alpha$ for each $i<\mu$.
				\end{enumerate}
				
			\item There are a partial type $p'(x,y)$ over $\emptyset$ and a sequence 
				$(a_i:i<\omega)$ such that
				\begin{enumerate}
					\item $(a_i:i<\omega)$ is $A$-indiscernible.
					\item $\bigcup_{i\in\omega}p'(x,a_i)$ is inconsistent.
					\item $p'(x,a_i)\vdash p(x)$ for each $i<\omega$.
					\item $DU(p'(x,a_i))\geq\alpha$ for each $i<\omega$.
				\end{enumerate}
				
			\item There are a partial type $p'(x,y)$ over the same set of parameters $A$ and a sequence
			  $(a_i:i<\omega)$ such that
				\begin{enumerate}
					\item $(a_i:i<\omega)$ is $A$-indiscernible.
					\item $\bigcup_{i\in\omega}p'(x,a_i)$ is inconsistent.
					\item $p'(x,a_i)\vdash p(x)$ for each $i<\omega$.
					\item $DU(p'(x,a_i))\geq\alpha$ for each $i<\omega$.
				\end{enumerate}
		\end{enumerate}
	\end{proposition}
	\begin{proof} \
	    \item[$1\Rightarrow2$.] 
			If $DU(p)\geq\alpha+1$ there exist $q(x)$ extending $p(x)$, dividing over $A$ with $DU(q)\geq\alpha$.
			Let $\psi(x,a)\in q$ dividing over $A$. So, there exist a sequence $(a_i:i<\omega)$ indiscernible over
			  $A$ with $a_0=a$ such that $\{\psi(x,a_i):i<\omega\}$ is inconsistent. By  point \emph{1} in
			  Remark~\ref{DUBasic}, $DU(p(x)\cup\psi(x,a))\geq\alpha$. 
			By conjugation, conditions $(c)$ is satisfied.

		\item[$2\Rightarrow3$.] 
			We can extend the indiscernible sequence to an indiscernible sequence of length $\lambda$.
			This sequence satisfies the required conditions.

		\item[$3\Rightarrow4$.] 
			Immediate.

		\item[$4\Rightarrow5$.] 
			Let $p(x)=p(x,a)$, where $p(x,y)$ is without parameters and $a$ enumerates $A$ (we assume the variables
			  $y$ in $p(x,y)$ and $\psi(x,y)$ are the same).
			Then  $p'(x,y)=p(x,y)\cup\{\psi(x,y)\}$ and $(b_i=aa_i:i<\mu)$ satisfy \emph{5}.

		\item[$5\Rightarrow6$.] 
		    Choose an infinite subsequence of $(a_i:i\in\mu)$ with all elements having the same type over $A$. 
		    Then apply the standard lemma (Lemma 7.1.1 in Tent, Ziegler\cite{Ziegler}) to  obtain a sequence
		      $(a'_i:i\in\omega)$ indiscernible over $A$ and satisfying the Ehrenfeucht-Mostowski type of the
		      subsequence. 
		    Then  $(a'_i:i\in\omega)$ satisfy the conditions of \emph{6}.
 
		\item[$6\Rightarrow7$.] 
			Immediate.

		\item[$7\Rightarrow1$.] 
			The closure under conjunction of $p'(x,a_0)$ divides over $A$, extends $p(x)$ and has $DU$-rank at least
			  $\alpha$. Therefore $DU(p)\geq\alpha+1$.
	\end{proof}
	
	Now we want to see that $DU$ may be characterized by the existence of certain trees of formula with certain
	  properties.
	\begin{definition}
		We define recursively a rooted tree $T_{\alpha,\lambda}$ for every ordinal $\alpha$ and cardinal $\lambda$:
		\begin{enumerate}
			\item $T_{0,\lambda}$ is a tree with a unique node.
			\item For an ordinal $\alpha+1$, we take $\lambda$ disjoint copies of $T_{\alpha,\lambda}$ and add a new
			  node related with all nodes, that is, a new root.
			\item For a limit ordinal $\alpha$, we take a disjoint union of all trees
			    $\{T_{\beta,\lambda}:\beta\in\alpha\}$ and add a new node related with all nodes, that is, a new
			    root. 
			  The node added in this step will be called a limit node of the tree.
		\end{enumerate}
	\end{definition}
	
	\begin{remark}
	    It is immediate that every $T_{\alpha,\lambda}$ is a tree.
		That is, the binary relation $R$ defined in the tree is a strict partial order (irreflexive and transitive)
		  and for each node $t$, the set $\{s:sRt\}$ is well-ordered.
	\end{remark}
	We use standard tree terminology: 
	  we say that a node $s$ is a child of a node $r$ (or $r$ is the parent of $s$) if $rRs$ and there are no nodes
	  $t$ with $sRt$ and $tRs$. The root of the tree will be the minimum. An  end-node is a node without children.
	We will denote by $F_{\alpha,\lambda}$ the set of parent nodes in $T_{\alpha,\lambda}$ which are not limit.
	$P_{\alpha,\lambda}$ will denote the set of nodes of $T_{\alpha,\lambda}$ which are a child of a non-limit. 
	
	Next lemma characterizes the value of $DU$ using the trees defined above. 
	Compare to the definition of the rank $DD$ in C\'ardenas, Farr\'e\cite{CardenasFarre1}.
	\begin{lemma}\label{DUeqTree}
	    Let $p(x)$ be a partial type over $A$ in $T$, $\alpha$ and ordinal and $\mu=\left(2^{|T|+|A|}\right)^+$.
        The following are equivalent:
		\begin{enumerate}
			\item $DU(p)\geq\alpha$.
			\item There is a sequence of formulas  $( \varphi_s(x,y_n) : s\in F_{\alpha,\mu})$, a sequence of
			  numbers  $( k_s : s\in F_{\alpha,\mu})$ and a sequence of parameters  
			  $( a_s : s\in P_{\alpha,\mu})$  such that
				\begin{enumerate}
					\item For every $s\in F_{\alpha,\mu}$, the set of formulas $\{\varphi_s(x,a_t):t$ is a child
					  of $s\}$ is $k_s$-inconsistent.
					\item For every end-node $s$, the set of formulas $p(x)\cup\{\varphi_{s}(x,a_{r}):tRs,\ r$ a
					  child of $t\}$ is consistent.
				\end{enumerate}
		\end{enumerate}
	\end{lemma}
	\begin{proof}
	    It is easily proved by induction using the equivalence \emph{4} in Proposition~\ref{EqvsDU}.
	\end{proof}

    \begin{proposition}
        Let $p(x)$ be a partial type over $A$.
        Then, there exists a set of parameters $B\subseteq A$ such that $|B|\leq|T|^{|DU(p)|}$ and
          $DU(p\upharpoonright B)=DU(p)$.
    \end{proposition}
    \begin{proof}
        We may assume $DU(p)<\infty$ and fix $\alpha=DU(p)+1$. For every partial type $q(x)$ over $A$ consider the
          type $\Sigma_{q,\overline{\varphi},\overline{k}}$ in the variables $(y_s : s\in P_{\alpha,\mu})$ 
          expressing the conditions \emph{(a)} and \emph{(b)} of Lemma~\ref{DUeqTree}.  
        Here $\overline{\varphi}=( \varphi_s(x,y_n) : s\in F_{\alpha,\mu})$ and 
          $\overline{k}=( k_s : s\in F_{\alpha,\mu})$ denote sequences of formulas and numbers and
          $\mu=\left(2^{|T|+|A|}\right)^+$.
        That is, $DD(q)<\alpha$ if and only if for every $\overline{\varphi}$ and $\overline{k}$,
          $\Sigma_{q,\overline{\varphi},\overline{k}}$ is inconsistent.
                
        As $DD(p)<\alpha$, for every $\overline{\varphi}$ and $\overline{k}$, by compactness, there is some finite
          $A_{\overline{\varphi},\overline{k}}\subseteq A$ such that
          $\Sigma_{p\upharpoonright A_{\overline{\varphi},\overline{k}},\overline{\varphi},\overline{k}}$ is
          inconsistent.
        Taking $B=\bigcup_{\overline{\varphi},\overline{k}} A_{\overline{\varphi},\overline{k}}$ we get
          $\Sigma_{p\upharpoonright B,\overline{\varphi},\overline{k}}$ is inconsistent for every
          $\overline{\varphi},\overline{k}$.
        We are using that $p\subseteq q$ implies 
          $\Sigma_{p,\overline{\varphi},\overline{k}}\subseteq \Sigma_{q,\overline{\varphi},\overline{k}}$.
    \end{proof}
    
	\begin{proposition}\label{DUInfty}
    	Let $p(x)$ be a partial type over $A$ such that $DU(p)=\infty$.
    	Then there exists a partial type $q(x)$ such that $p\subseteq q$, $q$ divides over $A$ and
    	  $DU(q)=\infty$.
	\end{proposition}
	\begin{proof}
    	For each $\alpha$, there is a $p_\alpha$ such that $p_\alpha\vdash\varphi_\alpha$ with $\varphi_\alpha$
    	  dividing over $A$, $p\subseteq p_\alpha$ and $DU(p_\alpha)\geq\alpha$.
		We may assume all formulas $\varphi_{\alpha}$ are conjugate over $A$.
		This is true because there are only boundedly many formulas and boundedly many types over $A$.
		
  		By conjugation over $A$ we may assume all $p_\alpha$ contain a formula that divides over $A$.
		So, $q=\bigcap p_\alpha$ is a partial type dividing over $A$.
		Then, $q(x)$ is a dividing extension of $p(x)$ with $DU(q)=\infty$.
	\end{proof}

\section{Relation between DU and other ranks}
	The $SU$-rank has traditionally been defined as the foundation rank of the forking relation.
	In the same way, we can define the rank $SU^d$ using dividing instead of forking. 
	Namely, $SU^d$ will be the foundation rank of the relation of dividing extension between complete types.
	To avoid confusion we will write $SU^f$ to refer to the ordinary rank $SU$ for forking.
	Obviously $SU^f(p)\geq SU^d(p)$.

	Now, we are going to see that we can define the known $D$-rank (for their definitions and properties, see for
	  example, Casanovas\cite{Casanovas11}) for formulas, as a foundation rank. 
	More precisely, as the foundation rank  of the relation of dividing between pairs $(\varphi,A)$ of formulas and
	  set of parameters satisfying $dom(\varphi)\subseteq A$. Using Lemma~\ref{IncDiv} one can easily show that $D$
	  does not depend on the set of parameters. 
	We can define similarly $D^f$ using the forking relation instead of dividing.
	Later, we will check (Proposition~\ref{DUeqDUf}) that both ranks are the same and therefore $D^f$ does no depend
	  on the set of parameters.

	\begin{definition}
	    $D$, $D^f$, $SU^d$ and $SU^f$ are the foundation ranks of the following relations $R_{dd}$, $R_{df}$,
	      $R_{sd}$ and $R_{sf}$: 
	    \begin{itemize}  
		    \item $(\varphi(x),A)R_{dd}(\psi(x),B)\text{ if and only if }
				  \models\psi\rightarrow\varphi\text{ and }\psi\text{ divides over }A$
			\item $(\varphi(x),A)R_{df}(\psi(x),B)\text{ if and only if }
				  \models\psi\rightarrow\varphi\text{ and }\psi\text{ forks over }A$
			\item $p(x)R_{sd}q(x)\text{ if and only if }q\text{ is a dividing extension of }p$
			\item $p(x)R_{sf}q(x)\text{ if and only if }q\text{ is a forking extension of }p$
		\end{itemize}	  
	    where $\varphi$ is a formula over $A$, $\psi$  is a formula over $B$ and $p$ and $q$ are complete types. 
	\end{definition}

    It is not difficult to verify that this definition of $D$ for formulas coincides with the traditional
      definition.
    For indeed, we can proceed as in Proposition~\ref{EqvsDU}.

	Next remark states well known properties of $SU^f$ (and therefore, of $SU$ in the context of simple
	  theories where $SU^d$ and $SU^f$ coincide). 
	We can check that $SU^d$ satisfy them in any theory. 
	The proofs are similar to the proofs for $DU$ in Remark~\ref{DUBasic}.
	\begin{remark}
		Let $p(x)\in S(A)$ and $q(x)\in S(B)$.
		The rank $SU^d$ satisfies:
		\begin{enumerate}
			\item If $q\subseteq p$ then $SU^d(p)\leq SU^d(q)$.
			\item For every $r$ completation of $p\vee q$, $SU^d(r)\leq\max(SU^d(p),SU^d(q))$.
			\item $SU^d(p)=0$ if and only if $p$ is algebraic.
		\end{enumerate}
	\end{remark}

    It is easy to verify that $D$ and $DU$ coincide for formulas:
	\begin{lemma}
		For every formula $\varphi(x)$, we have $D(\varphi)=DU(\varphi)$.
	\end{lemma}
	\begin{proof}
        We only need to prove $DU(\varphi)\leq D(\varphi)$.
        A proof by induction reduces the problem to show  $DU(\varphi)\geq\alpha+1$ implies
          $D(\varphi)\geq\alpha+1$.
        Assume $\varphi$ is over $A$ and $DU(\varphi)\geq\alpha+1$. 
        Then, there exists a partial type $q$ such that $\varphi\in q$, $q$ divides over $A$ and $DU(q)\geq\alpha$.
        Assuming $q$ closed under conjunction, there exists a formula $\psi\in q$ such that $\psi$ divides
          over $A$. 
        Obviously $\varphi\wedge\psi$ also divides over $A$ and $DU(\varphi\wedge\psi)\geq\alpha$. By the
          induction hypothesis, $D(\varphi\wedge\psi)\geq\alpha$.
        So, $D(\varphi)\geq\alpha+1$.
	\end{proof}

    A variation of the proof above also shows $D^f=DU^f$ for formulas. 
    Now, we are going to prove that $D^f$ and $DU^f$ are the same as $D$ and $DU$ respectively (and therefore do not
      depend on the set of parameters). So, from now on, we will use only $D$ and $DU$.
    \begin{proposition}\label{DUeqDUf}
		Let $p$ a partial type and $\varphi$ a formula both over $A$.
		Then, 
		\begin{enumerate}
			\item $DU(p)=DU^f(p,A)$.
			\item $D(\varphi)=D^f(\varphi,A)$.
		\end{enumerate}
	\end{proposition}
    \begin{proof}
        To prove \emph{1} it suffices to show that $DU(p)\geq DU^f(p)$. A proof by induction reduces to prove the
          following: $DU^f(p)\geq\alpha+1$ implies  $DU(p)\geq\alpha+1$, assuming it is true for $\alpha$.
        If $DU^f(p)\geq\alpha+1$, there exists $q\supseteq p$ such that $q$ forks over $A$ and $DU^f(q)\geq\alpha$
          and by the induction hypothesis, $DU(q)\geq\alpha$.
        Then, there exists $\{q_i:i\in n\}$ such that $q\equiv\bigvee_i q_i$ with each $q_i$ extending $q$ and
          dividing over $A$.
        Then, $DU(q)=\max\{DU(q_i):i\in n\}$.
        So, for some $q_i$, $DU(q_i)\geq\alpha$ and therefore, $DU(p)\geq\alpha+1$.

        \emph{2} follows from \emph{1}, since  $D^f=DU^f$ for formulas.
	\end{proof}

    From that proposition is immediate deduce that $SU^d(p)\leq SU^f(p)\leq DU(p)$ for any complete type $p$.
    
    $D$ is extended in a standard way to partial types $p$ as follows: 
	    \[D(p)=min\{D(\varphi):\varphi\text{ is a finite conjunction of formulas in }p\}\]
	As $D=DU$ for formulas, it is obvious that $DU(p)\leq D(p)$ for a partial type $p$, 
	  but in some cases they are not equal.
	In the next example we even see how $D$ can be $\infty$ while $DU$ not.
	
    \begin{example}\label{DnotequalDU}
		Let the language contain an infinite set of disjoint unary predicates $\{Q_i:i\in\omega\}$ and binary
		  relations $\{\leq_i:i\in\omega\}$.
		Each $\leq_i$ being a dense linear order without endpoints defined in $Q_i$. 
		Let $p$ denote $\{\neg Q_i(x):i\in\omega\}$.
		Then $DU(p)=1$ while $D(p)=\infty$.
    \end{example}
    \begin{proof}
        As $p$ is not algebraic, $DU(p)\geq1$. 
        Suppose $DU(p)\geq2$. 
        Then, by the equivalence \emph{4} in Proposition~\ref{EqvsDU}, there exist $\varphi(x,y)$ and
          $(a_i:i\in\mu)$ such that for each $i\in\mu$, $DU(p\cup\{\varphi(x,a_i)\})\geq1$ and
          $\{\varphi(x,a_i):i\in\mu\}$ is $k$-inconsistent for some $k$.
        Here $\mu=\left(2^{|T|+|A|}\right)^+$. 
        Any two realizations of $p$ different from $a_i$ have the same type over $a_i$, 
          so any realization of $p$ except maybe $a_i$ satisfy $\varphi(x,a_i)$. 
        This shows that $\{\varphi(x,a_i):i\in\mu\}$ is realized by every realization of $p$, except maybe
          $\{a_i:i\in\mu\}$ and therefore $\{\varphi(x,a_i):i\in\omega\}$ is not $k$-inconsistent. 
        This shows $DU(p)=1$.
        
        For each fine subset $S\subseteq I$, we will check that $D(\bigwedge_{i\in S}\neg Q_i)=\infty$, 
          so $D(p)=\infty$.
        Fix $j\in\omega-S$ and choose $\{a_i,b_i:i\in\omega\}$ in  $Q_j$ such that
          \[a_0<a_1<\ldots<a_n<\ldots<b_n<\ldots<b_1<b_0\]
        Then, the formula $a_n<x<b_n$ divides over $\{a_0b_0,\ldots a_{n-1}b_{n-1}\}$, so there is an infinite
          dividing sequence of formulas and therefore (see 14.3.3 Casanovas\cite{Casanovas11}) 
          $D(\bigwedge_{i\in S}\neg Q_i)=\infty$.
	\end{proof}

	In some cases $DU$ and $SU^d$ coincide for complete types:
	\begin{remark}\label{DUExt}
		Assume $DU$ has extension, i.e. for every partial type $p(x)$  over $A$, 
		  there exists $q(x)\in S(A)$ such that $p\subseteq q$ and $DU(p)=DU(q)$. 
		Then for every complete type $p$ $DU(p)=SU^d(p)$.
	\end{remark}
	\begin{proof}
		We prove that $SU^d(p)\geq DU(p)$ by induction on $\alpha$. 
		Let $p\in S(A)$ such that $DU(p)\geq\alpha+1$. 
		Then, there exists $q$ over $B$ such that $p\subseteq q$, $q$ divides over $A$ and $DU(q)\geq\alpha$.
		By the extension property, there exists $q'\in S(B)$ such that $q\subseteq q'$ and $DU(q')\geq\alpha$.
		By the induction hypothesis, $SU^d(q')\geq\alpha$ and therefore $SU^d(p)\geq\alpha+1$.
	\end{proof}

    In addition to the other mentioned ranks, we are going also to explore the relations with the $DD$-rank defined
      in C\'ardenas, Farr\'e\cite{CardenasFarre1}.
    In that paper we can found a definition of $DD$ from Shelah trees and several equivalences. 
    Here, we define $DD$ by dividing chains of complete types, which is the equivalence that we are going to use.
    \begin{definition}
        Let $p$ be a partial type over $A$.
        A \textbf{dividing chain of partial types of depth $\alpha$ in $p$} is a sequence of
          partial types $(p_i(x):i\in\alpha)$ and a sequence of sets of parameters $(A_i:i\in\alpha)$, each
		  $p_i$ a partial type over $A_i$, $p\subseteq p_0$, $A\subseteq A_0$, $p_0$ divides over $A$ and for
		  every $0<i<\alpha$, $p_{<i}\subseteq p_i$, $A_{<i}\subseteq A_i$ and $p_i$ divides over $A_{<i}$.

        If there is not such dividing chain for $p(x)$ we set $DD(p)=0_+$.	  
        Let $\beta$ the supremum of all possible depths of dividing chains of partial types in $p$
	    If this supremum does not exist we write $DD(p)=\infty$. 
		Otherwise, if $\beta$ is attained we put $DD(p)=\beta_+$ and $DD(p)=\beta_-$ if it is not attained.
		We call $DD(p)$ the \textbf{Dividing Depth of $p$}.

		In C\'ardenas, Farr\'e\cite{CardenasFarre1} is shown that $DD$ does not depend of the set of parameters.  
    \end{definition}

	Ranks $DU$, $SU^d$ and $SU^f$ take the value $\infty$ at the same time:
	\begin{proposition}\label{SU,DU are infinite same time}
		Let $p$ be a partial type. 
		Then, $DD(p)\geq\omega_+$ if and only if $DU(p)=\infty$.
		Moreover, if $p$ is complete, then $SU^d(p)=\infty$, $SU^f(p)=\infty$ are also equivalent to $DU(p)=\infty$.
	\end{proposition}
	\begin{proof}
        The first equivalence has a standard proof based in properties of foundations ranks (see for example
          Remark~\emph{13.6} in Casanovas\cite{Casanovas11}) and Proposition~\ref{DUInfty}.

        For the second equivalence, assume $DU(p)=\infty$. 
        By Lemma~\ref{DUInfty}, we can build a dividing chain of partial types $(p_i:i\in\omega)$ and sets of
          parameters $(A_i:i\in\omega)$ such that $p=p_0$, $A=A_0$ and for every $i\in\omega$, 
          $p_i\subseteq p_{i+1}$, $A_i\subseteq A_{i+1}$, $p_{i+1}$ divides over $A_i$.
        Let $a\models\bigcup_{i\in\alpha}p_i$.
        Then $(tp(a/A_i):i\in\omega)$ is a dividing chain of complete types. 
        It is easy to check by induction over $\alpha$ that for every $i\in\omega$, $SU^d(a/A_i)\geq\alpha$.
	\end{proof}

    In the following two results, in order to compare $DD$ with other ranks, we will suppress the subscripts in the values of $DD$. 

	\begin{proposition}\label{DD=DU when finite}
		Let $p$ a partial type with $DD(p)$ finite, then $DD(p)=DU(p)$.
		Moreover, if $p$ is complete, then $DD(p)=SU^d(p)=SU^f(p)=DU(p)$.
	\end{proposition}
	\begin{proof}
		If $DU(p)\geq n$ we can build a dividing chain of partial types of length $n$, $(p_i:i<n)$ and $(A_i:i<n)$.
		So, $DD(p)\geq n$. 
		The converse is immediate.
		If moreover $p$ is complete, let $a\models\bigcup_{i<n}p_i$.
		Then, the sequence $(tp(a/A_i):i<n)$ forms a dividing chain of complete types witnessing $SU^d(p)\geq n$.
	\end{proof}

    The inequality $SU^f(p)\leq D(p)$ is well known (see Kim\cite{Kim}) but a standard proof needs simplicity of the
      theory and does not work in full generality. 
    Actually, it is true in general:
	\begin{remark}\label{SUd_SUf_DU}
		Let $p$ be a partial type. 
		Then, $DD(p)\leq DU(p)\leq D(p)$.
		Moreover if $p$ is a complete type, $DD(p)\leq SU^d(p)\leq SU^f(p)\leq DU(p)\leq D(p)$.
	\end{remark}
	\begin{proof}
		It is immediate by Proposition~\ref{DUeqDUf} that $SU^d(p)\leq SU^f(p)\leq DU(p)\leq D(p)$.
		And by Proposition~\ref{SU,DU are infinite same time} and Proposition~\ref{DD=DU when finite} we obtain
		  $DD(p)\leq SU^d(p)$ and $DD(p)\leq DU(p)$.
	\end{proof}

    We are going to use a result about $DD$ in C\'ardenas, Farr\'e\cite{CardenasFarre1} to prove that $DU$,
      $SU^d$ and $SU^f$ have a bounded number of different values.
    We take next lemma from Proposition~\emph{3.10} in C\'ardenas, Farr\'e\cite{CardenasFarre1}:
    \begin{lemma}
        Let $p(x)$ be a partial type over $A$. 
        Then, there exists a set of parameters $B\subseteq A$ such that $|B|\leq|T|^{|DD(p)|}$ and
          $DD(p\upharpoonright B)=DD(p)$.
    \end{lemma}

	\begin{proposition}\label{DUMax}
	    There is some ordinal $\alpha$ such that $DU(p)\geq\alpha$ implies $DU(p)=\infty$.
	\end{proposition}
	\begin{proof}
		Observe that, as $DU$ takes the same values over conjugate sets of parameters and there are boundedly many
		  non-conjugate sets of parameters of size  $\leq |T|^{\aleph_0}$, the $DU$-values on types over a set
		  of parameters of size  $\leq |T|^{\aleph_0}$ is upper bounded.
		By Proposition~\ref{SU,DU are infinite same time} and the previous lemma, for any partial type $p$ over $A$
		  with $DU(p)<\infty$ there is some $B\subseteq A$ with $|B|\leq |T|^{\aleph_0}$ and 
		  $DU(p)\leq DU(p\upharpoonright B)<\infty$. 
		Therefore the set of non-infinite values taken by $DU$  is bounded. 
	\end{proof}

    \begin{remark}
        The same $\alpha$ as in Proposition~\ref{DUMax} satisfies that 
        $SU^d(p)\geq\alpha$ implies $SU^d(p)=\infty$ and $SU^f(p)\geq\alpha$ implies $SU^f(p)=\infty$.
    \end{remark}
	
	We have seen so far that $D=D^f$ and $DU^d=DU^f$, but we do not know if, in general, $SU^d=SU^f$ or $SU^f=DU$.
	We know that the three ranks are equal for finite values and the value $\infty$, but in all intermediate cases,
	  when $DD(p)=\omega_-$, we do not have the answer. 
	So, we have these two open questions:
	\begin{question}
	    Is there a complete type $p$ such that $SU^d(p)<SU^f(p)$?
	\end{question}

	\begin{question}
	    Is there a complete type $p$ such that $SU^f(p)<DU(p)$?
	\end{question}

    In C\'ardenas, Farr\'e\cite{CardenasFarre3} we prove that in an $NTP_2$ theory, for any stable complete type
      $p$, $SU^d(p)=SU^f(p)$.
    We also prove that if $SU^d$ has extension then $SU^d=SU^f$.

\section{Supersimple types}
	The following are two equivalent definitions of a simple type (see  in Hart, Kim, Pillay\cite{HKP} and
	  Chernikov\cite{Chernikov}).
	\begin{definition}
    	Let $p(x)$ be a partial type over $A$. 
    	$p$ is \textbf{simple} if and only if one of the following two equivalent conditions are satisfied:
    	\begin{enumerate}
            \item for every $B\supseteq A$ and every realization $a$ of $p(x)$, there is some $B_0\subseteq B$ with
              $|B_0|<|T|^+$ such that $a\ind^d_{B_0}B$.
    	    \item for every $B\supseteq A$ and every realization $a$ of $p(x)$, there is some $B_0\subseteq B$ with
    	      $|B_0|<|T|^+$ such that $a\ind^d_{AB_0}B$.
    	\end{enumerate}
	\end{definition}
	
	From this, one might think in defining a supersimple type in two different ways, replacing in both definitions
	  the bound $|T|^+$ by $\aleph_0$.
	In fact, in Hart, Kim, Pillay\cite{HKP}, they suggest to define a supersimple type through the first
	  alternative, although they do not develop the implications of this possibility.
		
	We will see through the Example~\ref{exampleSupersimple} that these possible definitions are not equivalent and
	  that the first one depends on the set of parameters while by Corollary~\ref{EqsSupersimple} the second not.
	In addition, in this example we show also a superstable complete type not satisfying the first possible
	  definition of supersimple.	
	All of that indicate us that the correct way of defining a supersimple type will be the second:
	\begin{definition}
		Let $p(x)$ be a partial type over $A$.
		$p$ is \textbf{supersimple} if and only if for every $B\supseteq A$ and every realization $a$ of
		  $p(x)$, there exists a finite set $B_0\subseteq B$ with $a\ind^d_{AB_0}B$.
	\end{definition}

	\begin{remark}
	    It is obvious that the discarded definition of supersimple type implies our definition of supersimple type.
	\end{remark}
	
	The notion of supersimple type satisfies the following expected properties:
	\begin{remark}\label{basicPropSS}
		The following are satisfied:
		\begin{enumerate}
			\item If $p(x)$ is supersimple and $p(x)\subseteq q(x)$, then $q(x)$ is supersimple.
			\item If $p(x,y)$ is supersimple, then the type $\exists yp(x,y)$ is supersimple.
			\item $tp(ab/A)$ is supersimple if and only if $tp(a/A)$ and $tp(b/Aa)$ are supersimple. 
				More generally, $tp((a_i:i\in n)/A)$ is supersimple if and only if $tp((a_i:i\in n)/A)$ is
				  supersimple.
			\item Assume that $x$ and $y$ are disjoint. 
			  $p(x)$ and $q(y)$ are supersimple if and only if $p(x)\cup q(y)$ is supersimple.
			\item Let $p(x)$ be over $A$. 
			  Then $p$ is supersimple if and only if every $q(x)\in S(A)$ extending $p(x)$ is supersimple.
 			\item $p_1(x), p_2(x)$ are supersimple if and only if $p_1\vee p_2$ is supersimple.
		\end{enumerate}
	\end{remark}
    \begin{proof} 
        \emph{1} is obvious assuming $p$ and $q$ are over the same set of parameters.
        
        \emph{2}. 
            Let $a\models\exists yp(x,y)$, then  $ab\models p(x,y)$ for some $b$. Let $B\supseteq A$. 
            As $p(x,y)$ is supersimple, there exists a finite $B_0\subseteq B$ with  $ab\ind^d_{AB_0}B$. 
            So, $a\ind^d_{AB_0}B$.
        
        \emph{3}$\Rightarrow$). 
        	$tp(a/A)=\exists ytp(ab/A)$ and $tp(b/A)=\exists xtp(ab/A)$ are supersimple by \emph{2} and
        	  therefore by \emph{1} $tp(b/Aa)$ is also supersimple.
        
        \emph{3}$\Leftarrow$).
        	Let $B\supseteq A$ and $a'b'\equiv_Aab$.
            By the first condition, as $a'\equiv_Aa$, there exists $B_1\subseteq B$ finite such that
              $a'\ind^d_{AB_1}B$.
            As $tp(b'/Aa')$ is supersimple, there exists $B_2\subseteq B$ finite such that $b'\ind^d_{Aa'B_2}Ba'$.
        	 Now, taking $B_0=B_1B_2$ we have $a'\ind_{AB_0}B$ and $b'\ind_{Aa'B_0}B$. 
        	 By left transitivity we obtain $a'b'\ind_{AB_0}B$.			
        
        \emph{4}.  
            Assume $p(x)$ and $q(y)$ are supersimple over $A$ and let $B\supseteq A$ and $ab\models p(x)\cup q(y)$.
            Then $tp(a/A)$ and $tp(b/Aa)$ are supersimple  by \emph{1}. 
            By \emph{3}, $tp(ab/A)$ is supersimple. 
            So, there is some finite $B_0\subseteq B$ with $ab\ind^d_{AB_0}B$.
        
        \emph{5}$\Rightarrow$) is  trivial by $1$.
        
        \emph{5}$\Leftarrow$). Let $B\supseteq A$ and $a\models p$. 
            As $tp(a/A)$ is supersimple, there exists $B_0\subseteq B$ finite such that $a\ind^d_{AB_0}B$.
        				
        \emph{6} follows from \emph{5}, since any completion of $p\lor q$ is either a completion of $p$ or a
          completion of $q$.
    \end{proof}

	We remember the following result for $DD$ from C\'ardenas, Farr\'e\cite{CardenasFarre1}:
	\begin{lemma}\label{DDEqCardinals2}
		Let $p$ be a partial type over a set of parameters $A$.
		Let $\kappa$ be any regular cardinal number.
		The following are equivalent: 
		\begin{enumerate}
			\item $DD(p)<\kappa_+$.
			\item For every $B\supseteq A$ and $a\models p$, there exists a set $B_0\subseteq B$ with $|B_0|<\kappa$
			  such that $a\ind^d_{AB_0}B$.
		\end{enumerate}
	\end{lemma}

	\begin{corollary}\label{EqsSupersimple}  
    	The definition of supersimple does not depend on the set of parameters. 
    	Moreover, the following are equivalent for a partial type $p$ over $A$: 
    	  \emph{1}. $p$ is supersimple, 
    	  \emph{2}. $DD(p)<\omega_+$,
    	  \emph{3}. $DU(p)<\infty$,
    	  \emph{4}. For every completion $q\in S(A)$ of $p$, $SU^d(q)<\infty$.
    	  \emph{5}. For every completion $q\in S(A)$ of $p$, $SU^f(q)<\infty$.
	\end{corollary}
	\begin{proof}
	    By previous Lemma, the fact that $DD(p)$ does not depend on the set of parameters, Remark~\ref{SUd_SUf_DU}
	      and Proposition~\ref{SU,DU are infinite same time}. 
	\end{proof}
	
    We can obtain other two equivalences of the notion of supersimple type replacing dividing by forking.
    We can define a forking chain of partial types in $p$ in a similar way to dividing chain (see C\'ardenas, Farr\'e\cite{CardenasFarre1}).

    \begin{remark}
        Let $p(x)$ be a partial type over $A$.
		$p$ is \textbf{supersimple} if and only it verifies the following equivalent conditions:
        \begin{enumerate}
            \item For every $B\supseteq A$ and every realization $a$ of $p(x)$, there exists a finite set
              $B_0\subseteq B$ with $a\ind^f_{AB_0}B$.
            \item There is not a forking chain of partial types in $p$ of depth $\omega$.
        \end{enumerate}
    \end{remark}
    \begin{proof}
        Arguing as in  Proposition~\ref{SU,DU are infinite same time} it follows that \emph{2} is equivalent to the fact that $DU^f(p)<\infty$. 
        Therefore, \emph{2} is equivalent to supersimple.
        The equivalence between \emph{1} and \emph{2} is similar to the proof of lemma~\ref{DDEqCardinals2}  changing dividing by forking. 
    \end{proof}

	We remember the definition of the Lascar rank $U$ and the definition of a superstable type of
	  Poizat\cite{Poizat}:
	\begin{definition}
		The \textbf{$U$-rank} for a complete type $p(x)\in S(A)$ is defined as follows:
		\begin{enumerate}
			\item $U(p)\geq 0$.
			\item $U(p)\geq\alpha+1$ if and only if for each cardinal number $\lambda$ there is a set $B\supseteq A$
			  and there are at least $\lambda$ many types $q(x)\in S(B)$ extending $p$ and such that $U(q)\geq\alpha$.
			\item $U(p)\geq\alpha$ with $\alpha$ a limit ordinal, if and only if $U(p)\geq\beta$ for all $\beta<\alpha$.
		\end{enumerate}
		$U(p)$ is the supremum of all $\alpha$ such that $U(p)\geq\alpha$.
		If such supremum does not exist we set $U(p)=\infty$.\par		
	\end{definition}

	\begin{definition}
		Let $p$ be a complete type. $p$ is \textbf{superstable} if and only if $U(p)<\infty$.
	\end{definition}

	In C\'ardenas, Farr\'e\cite{CardenasFarre3} we prove that a complete type is stable and supersimple if and only
	  if is superstable.

    \begin{example}\label{exampleSupersimple} 
        There is an example of a superstable and supersimple type $p\in S(A)$ not satisfying the discarded
          definition of supersimple. 
        However, for some $b$, $p$ considered over $Ab$ satisfies the alternative definition, so the discarded
          definition depend on the set of parameters.
    \end{example}
    \begin{proof}
        Consider the theory of infinitely many refining equivalence relations, which is a stable non supersimple
          theory with quantifier elimination.
        The language consists in $\omega$ equivalence relations $\{E_i:i\in\omega\}$, $E_0$ has infinite many
          classes, $E_{i+1}$ refines $E_i$ and each $E_i$-class is partitioned into infinitely many
          $E_{i+1}$-classes.	
        Given $d$ and $C$, denote $\rho(d/C)=\infty$ if $d\in C$, $\rho(d/C)=sup\{n:dE_nc\text{ for some }c\in C\}$
          otherwise. 
        Here we consider $\omega<\infty$. One can verify that $tp(d/BC)$ divides over $C$ if and only if
          $\rho(d/C)<\rho(d/BC)$. So,
        	\[D\ind_CB\Leftrightarrow\text{ for every }d\in D:\rho(d/C)=\rho(d/BC)\]
        
        Now we choose and fix $a$ and $b$ such that $aE_ib$ for every $i\in\omega$ and take $A=\{a_i:i\in\omega\}$
          such that  $aE_ia_i$ and $a\cancel{E_{i+1}}a_i$ for every $i\in\omega$. 
        The example is $p=tp(a/A)$.   
        
        $p$ is supersimple: Given $B\supseteq A$ and  $a'\models p$, $\rho(a'/A)=\omega$ and $\rho(a'/B)\geq\omega$.
        So, taking $B_0=\{a'\}\cap B$, we have $a'\ind_{AB_0}B$.
        
        $p$ is superstable: $p$ is not algebraic, so $U(p)\geq 1$. On the other hand, it is immediate to check that
          for any parameter set, $p$ only has a non algebraic extension, so applying the definition of $U$,
          $U(p)\not\geq2$.
        
        $p$ does not satisfy the discarded definition of supersimple but $p$ considered as a partial type over
          $Ab$ satisfies the discarded definition of supersimple: for every finite $B_0\subseteq A$, we have
          $\rho(a/B_0)<\omega$,  
        So, $a\nind_{B_0}A$. 
        But given $B\supseteq Ab$ and $a'\models p$, we have $\rho(a'/b)=\omega$ and $\rho(a'/B)\geq\omega$. 
        So, taking $B_0=\{a'b\}\cap B$, we have $a'\ind_{B_0}B$.
    \end{proof}

	Although for a particular type the discarded definition is not equivalent to supersimplicity, for a fixed theory
	  the fact that all types satisfy one of the definitions is equivalent to all types satisfy the other.
	\begin{remark}
		The following are equivalent:
		\emph{1}: $T$ is supersimple.
		\emph{2}: $\{x=x\}$ is supersimple.
		\emph{3}: Every complete type is supersimple.
		\emph{4}: Every partial type is supersimple.
		\emph{5}: For every $p\in S(A)$, there exists a finite subset $A_0\subseteq A$ such that $p$ does not
		  divide over $A_0$.
		\emph{6}: For every $p\in S(A)$, every $B\supseteq A$ and every realization $a$ of $p$, there exists a
		  finite subset $B_0\subseteq B$ such that $a\ind^d_{AB_0}B$.
	\end{remark}
	\begin{proof}
	    The equivalence between $3$ and $4$ follows from remark~\ref{basicPropSS}. 
	    The other are standard, see $13.1$, and $13.4$ in Casanovas\cite{Casanovas11}.
	\end{proof}

	Now we improve slightly for $SU^d$ and $SU^f$ the known fact that in simple theories $SU$ is preserved by
	  non-forking extensions.
	We recall that a theory is called \textbf{Extensible} if forking has existence, that is every complete type
	  does not fork over its parameter set.
	For instance, simple theories are extensible.

	\begin{proposition}\label{SUcharDiv}
		In a extensible theory, let $p(x)\in S(A)$ and $q(x)\in S(B)$ be such that $p(x)\subseteq q(x)$
		  and $tp(B/A)$ is simple. 
		If $q$ does not fork over $A$ then $SU^d(q)=SU^d(p)$ and $SU^f(q)=SU^f(p)$.
	\end{proposition}
	\begin{proof}
        We will prove that $SU^d(p)\leq SU^d(q)$. 
	    The proof is similar for $SU^f$.
		If $SU^d(q)=\infty$ the result is immediate, so we can assume without loss of generality that $q$ is
		  supersimple.
		We use induction on $\alpha$ to prove that if $SU^d(p)\geq\alpha$ then $SU^d(q)\geq\alpha$.
		This is clear for $\alpha=0$ or limit ordinal.
		Assume $SU^d(p)\geq\alpha+1$.
		Now, by the definition of $SU^d$, there is a dividing extension $p_1\in S(C)$ of $p$ such that
		  $SU^d(p_1)\geq\alpha$.
		
		Let $d\models q$ and $d'\models p_1$.
		As $d'\equiv_{A}d$,	there exists $C'$ such that $d'C\equiv_AdC'$.
		Using $T$ extensible we can choose $C''$ such that $C''\equiv_{Ad}C'$ and $C''\ind^f_{Ad}B$.
		
		As $d\ind^f_A B$ and $C''\ind^f_{Ad}B$, by left transitivity, $C''d\ind^f_A B$. 
		As $tp(B/A)$ is simple, using symmetry (Proposition~\emph{7.3} in Casanovas\cite{Casanovas11b}), 
		  $B\ind^f_A C''d$ and therefore $B\ind^f_{C''}d$. 
		Since $tp(d/B)$ and $tp(B/A)$ are simple, $tp(dB/A)$ is simple and therefore $tp(d/C'')$ is simple. 
		Using symmetry again $d\ind^f_{C''}B$.
		Since $tp(B/A)$ is simple $tp(B/C'')$ is  simple and therefore  $tp(C''B/C'')$ is also simple. 
		By induction hypothesis, $SU^d(d/C''B)\geq\alpha$. 
		By $B\ind^f_A C''d$ we get $B\ind^d_{A}C''$.
		With $d\nind^d_{A}C''$, we obtain $d\nind^d_BC''$ and therefore $tp(d/C''B)$ divides over $B$.
		So, finally $SU^d(q)\geq\alpha+1$.
	\end{proof}
	
	\begin{corollary}
		In an extensible theory, let $p$ be a complete type over $A$ and $q$ be a partial type over $B$ such that 
		  $q$ is a non forking extension of $p$ and $tp(B/A)$ is simple.
		If $q$ is supersimple then $p$ is supersimple.
	\end{corollary}

    By Proposition~\emph{3.9} in C\'ardenas, Farr\'e\cite{CardenasFarre1}, we have a similar corollary using $DD$
      with somewhat different hypotheses:
	\begin{corollary}
		Let $p$ be a complete type over $A$ and $q$ be a partial type over $B$ such that 
		  $q$ is a non forking extension of $p$ and $tp(B/A)$ is simple and co-simple.
		If $q$ is supersimple then $p$ is supersimple.
	\end{corollary}
	\begin{proof}
		If $q$ is supersimple, any completation $\bar{q}$ of $q$ is supersimple and $DD(\bar{q}\leq\omega_-$.
		By Proposition~\emph{3.9} in C\'ardenas, Farr\'e\cite{CardenasFarre1}, $DD(p)\leq\omega_-$ and therefore is
		  supersimple.
	\end{proof}

    It is immediate to conclude from Proposition~\emph{3.11} in C\'ardenas, Farr\'e\cite{CardenasFarre1} the following:
    \begin{corollary}
        Let $p(x)$ be a non-supersimple partial type over $A$. Then,
		\begin{enumerate}
            \item There exists $B\supseteq A$ and a completion $q\in S(B)$ of $p$ dividing over $A$ with $q$ non-supersimple.
		    \item If $M\supseteq A$ is an $|A|^+$-saturated model, there exists a completion $q\in S(M)$ of $p$ dividing over $A$ with $q$ non-supersimple.
		\end{enumerate}        
    \end{corollary}

\end{document}